\documentclass[a4paper,reqno,10.5pt, oneside]{amsart}

\usepackage{amssymb}
\usepackage{amstext}
\usepackage{amsmath}
\usepackage{amscd}
\usepackage{amsthm}
\usepackage{amsfonts}
\usepackage{enumerate}
\usepackage[dvipdfmx]{graphicx}
\usepackage{latexsym}
\usepackage{mathrsfs}
\usepackage{lineno}

\usepackage[dvipdfmx]{xcolor}

\usepackage{comment}

\newtheorem{thm}{Theorem}[section]

\newtheorem{cor}[thm]{Corollary}
\newtheorem{lem}[thm]{Lemma}
\newtheorem{prop}[thm]{Proposition}

\theoremstyle{definition}

\newtheorem{ex}[thm]{Example}

\newtheorem*{Notation}{Notation}

\newtheorem{rem}[thm]{Remark}

\makeatletter
  
  \@addtoreset{equation}{section}
\makeatother

\newcommand{\ZZ}{\mathbb{Z}}

\newcommand{\msum}{\mathsf{msum}}
\newcommand{\disc}{\mathsf{disc}}

\begin{document}
\title[$k$-consecutive sums]{Permutations with small maximal $k$-consecutive sums}
\author[A Higashitani \and K. Kurimoto]{Akihiro Higashitani \and Kazuki Kurimoto} 
\address[A.\,Higashitani]{Department of Pure and Applied Mathematics, 
Graduate School of Information Science and Technology, Osaka University, Osaka, Japan, 565-0871} 
\email{higashitani@ist.osaka-u.ac.jp}
\address[K. Kurimoto]{Department of Mathematics, Kyoto Sangyo University, Motoyama, Kamigamo, Kita-Ku, Kyoto, Japan, 603-8555}
\email{i1885045@cc.kyoto-su.ac.jp}

\maketitle

\begin{abstract} 
Let $n$ and $k$ be positive integers with $n>k$. 
Given a permutation $(\pi_1,\ldots,\pi_n)$ of integers $1,\ldots,n$, we consider $k$-consecutive sums of $\pi$, i.e., 
$s_i:=\sum_{j=0}^{k-1}\pi_{i+j}$ for $i=1,\ldots,n$, where we let $\pi_{n+j}=\pi_j$. 
What we want to do in this paper is to know the exact value of
$$\msum(n,k):=\min\left\{\max\{s_i : i=1,\ldots,n\} -\frac{k(n+1)}{2}: \pi \in S_n\right\},$$ 
where $S_n$ denotes the set of all permutations of $1,\ldots,n$. 
In this paper, we determine the exact values of $\msum(n,k)$ for some particular cases of $n$ and $k$. 
As a corollary of the results, we obtain $\msum(n,3)$, $\msum(n,4)$ and $\msum(n,6)$ for any $n$. 
\end{abstract}

\bigskip

\section{Introduction}

%
%

Let us fix some notation used throughout this paper. 
Let $n$ and $k$ be positive integers with $n>k$. Let $S_n$ denote the set of permutations with $n$ elements. 
Given any $\pi =(\pi_1,\ldots,\pi_n) \in S_n$, we always let $\pi_{n+i}=\pi_i$ for $i \geq 1$. 
We consider $k$-consecutive sums in $\pi$ starting from $\pi_i$, denoted by $s_i$, i.e., let $s_i=\sum_{j=0}^{k-1} \pi_{i+j}$. 
Note that the average of $s_1,\ldots,s_n$ is $k(n+1)/2$ since $\sum_{i=1}^ns_i=k\sum_{\ell=1}^n\ell$. Let 
$$\msum(\pi,k):=\max\{s_i : i=1,\ldots,n\} - \frac{k(n+1)}{2}.$$ 
The main object in this paper is 
$$\msum(n,k):=\min\{\msum(\pi,k) : \pi \in S_n\}. $$
Similarly, we also define \begin{align*}
\disc(\pi,k):=\max\left\{\left|s_i- \frac{k(n+1)}{2}\right| : i=1,\ldots,n\right\} \text{ and } \disc(n,k):=\min\{\disc(\pi,k) : \pi \in S_n\}. 
\end{align*}
Clearly, $\msum(n,k) \leq \disc(n,k)$. Remark that $\msum(n,k)$ does not necessarily coincide with $\disc(n,k)$ (see Remark \ref{rima-ku}). 

Moreover, by the definition of $\msum(n,k)$ and $\disc(n,k)$, we see that $\msum(n,k)>0$ and $\disc(n,k)>0$ for any $n$ and $k$. 
More precisely, we have 
\begin{equation}\label{chuui}\begin{split}
\msum(n,k) \text{ (and $\disc(n,k)$) }=\begin{cases}
\text{(odd positive integer)}/2, \;\;&\text{if $k$ is odd and $n$ is even}, \\
\text{positive integer}, &\text{otherwise}. \quad \end{cases}
\end{split}\end{equation}
In particular, we have $\msum(n,k) \geq 1/2$ if $k$ is odd and $n$ is even, and $\msum(n,k) \geq 1$ otherwise. 
Remark that the same inequalities hold for $\disc(n,k)$.

Furthermore, we see that $\msum(n,k)=\msum(n,n-k)$ for any $n$ and $k$ with $n > k$ (see Proposition \ref{equal}). 
Thus, for the investigation of $\msum(n,k)$, we may assume that $n \geq 2k$. 

\bigskip

The goal of the present paper is the computation of the exact values of $\msum(n,k)$ for small $k$'s (Corollary \ref{meidai}). 
For this goal, we give the exact values or some non-trivial lower bounds of $\msum(n,k)$ for some particular cases of $n$ and $k$ 
(see Theorems \ref{main1}, \ref{main2}, \ref{main3}, Proposition \ref{main4} and Corollary \ref{kei}). 

\bigskip


Let us explain the history of the computations of $\msum(n,k)$ and $\disc(n,k)$. 

Although it is unclear which literature treated $\msum(n,k)$ or $\disc(n,k)$ at first, 
the textbook of discrete mathematics by Liu \cite{Liu} mentions the computation of $\msum(36,3)$ as an exercise (as far as the authors know). 
Morris \cite{Mor} and Stefanovi\'c \cite{S} determined the exact values of $\msum(n,k)$ for some small $n$ and $k$. 
By Anstee--Ferguson--Griggs \cite{AFG}, the exact values of $\disc(n,k)$ were investigated in many cases of $n$ and $k$. 
In \cite{AFG}, $\disc(n,3)$ was determined for almost every $n$ as a corollary of the results obtained there. 
The main interest in \cite{AFG} was $\disc(n,k)$. After the paper \cite{AFG}, 
Stefanovi\'c and \u{Z}ivkovi\'c \cite{MM} studied $\disc(n,k)$ and $\msum(n,k)$ for some other cases of $n$ and $k$. 
By combining the results in \cite{AFG} and \cite{MM}, we can know the exact values of $\disc(n,3)$ for every $n$. 
In \cite{MM}, they also supplied the exact values of $\msum(n,k)$ for some particular cases. For more detail, see Section \ref{preliminary}. 
We collect the results of \cite{AFG} (resp. \cite{MM}) in Theorem \ref{thm:AFG} (resp. Theorem \ref{thm:MM}) 
and provide Corollary \ref{previous} which immediately follows from Theorems \ref{thm:AFG} and \ref{thm:MM}. 

A kind of generalizations of $\msum(n,k)$ is to treat a ``$2$-dimensional analogue'' of this problem. 
More precisely, Kawamura \cite{K} studies the difference of the maximum and the minimum of the sums of 
the integers in $k \times \ell$ regions in the arrangements of $1,\ldots,mn$ in $m \times n$ square board. 


\bigskip


As further contributions, we compute $\msum(n,k)$ for some more cases of $n$ and $k$. The main results of this paper are as follows: 
\begin{thm}\label{main1}
{\em (a)} Let $k$ be even. If $n \equiv \pm 1$ (mod $k$), then $\msum(n,k)=k/2$. \\
{\em (b)} Let $k$ be odd. 
\begin{itemize}
\item If $n \equiv \pm 1$ (mod $k$) and 
\begin{itemize}
\item $n$ is even, then $\msum(n,k)=k/2$; 
\item $n$ is odd, then $\msum(n,k)=(k+1)/2$. 
\end{itemize}
\item If $n \equiv 0$ (mod $k$) and 
\begin{itemize}
\item $n$ is even, then $\msum(n,k)=3/2$ when $n \geq 4k$ and $\msum(n,k)=1/2$ when $n = 2k$; 
\item $n$ is odd and $n \geq k(2k+3)$, then $\msum(n,k)=2$. 
\end{itemize}
\end{itemize}
\end{thm}
\begin{thm}\label{main2}
Let $n$ and $k$ be even. Then $\msum(n,k)=1$. 
\end{thm}
\begin{thm}\label{main3}
Let $k$ be odd with $k \geq 5$. Assume that $n \equiv (k+1)/2$ (mod $k$). Then $\msum(n,k) >1$. 
\end{thm}
\begin{prop}\label{main4}
Assume that $n \equiv 2$ (mod $5$). Then $\msum(n,5) > 1$. 
\end{prop}

A proof of Theorem \ref{main1} consists of three parts: 
\begin{enumerate}
\item the case $n \equiv 1$ (mod $k$) (see Section \ref{n=1}); 
\item the case $n \equiv -1$ (mod $k$) (see Section \ref{n=-1}); 
\item the case $n \equiv 0$ (mod $k$) (see Section \ref{n=0}). 
\end{enumerate}
A proof of Theorem \ref{main2} is given in Section \ref{nkeven}. 
Proofs of Theorem \ref{main3} and Proposition \ref{main4} are given in Section \ref{k=5}. 

\begin{cor}\label{kei}
Let $k$ be even. Assume that $n \equiv k/2$ (mod $k$). Then we have $\msum(n,k)=1$ and we also have $\disc(n,k)=1$ if $n$ is odd. 
\end{cor}
Corollary \ref{kei} follows from Theorem \ref{main2} together with some results in \cite{AFG}. 
We will prove this in Section \ref{preliminary} after presenting the results in \cite{AFG}. 
We remark that $\disc(n,k)=1$ does not hold in general if $n$ is even with $n \equiv k/2$ (mod $k$). See Proposition \ref{disc_msum}.

Collectively, the exact values of $\msum(n,3)$, $\msum(n,4)$, $\msum(n,5)$ and $\msum(n,6)$ 
(and some inequalities for $\msum(n,5)$) can be obtained from those results above together with some known results. 
\begin{cor}\label{meidai}
{\em (a) ($k=3$)} We have $$\msum(n,3)=\begin{cases}
1/2 \;\; &\text{ if }n = 6, \\
1 \;\; &\text{ if }n=9,15, \\
3/2 &\text{ if $n$ is even with $n \geq 8$}, \\
2 &\text{ if $n$ is odd with $n \geq 7$ all but $n =9,15$}. 
\end{cases}$$
{\em (b) ($k=4$)} We have $$\msum(n,4)=\begin{cases}
2 \;\; &\text{ if $n$ is odd}, \\
1 \;\; &\text{ if $n$ is even}. \end{cases}$$
{\em (c) ($k=5$)} We have $$\msum(n,5)=\begin{cases}
1/2 &\text{ if $n=10$}, \\
3/2 &\text{ if $n \equiv 0$ (mod $10$) with $n \geq 20$}, \\
2 &\text{ if $n \equiv 5$ (mod $10$) with $n \geq 65$}, \\
5/2 &\text{ if $n \equiv 4,6$ (mod $10$)}, \\
3 &\text{ if $n \equiv 1,9$ (mod $10$)}, 
\end{cases} \text{ and }
\msum(n,5) \geq \begin{cases}
3/2 &\text{ if $n \equiv 2,8$ (mod $10$)}, \\
2 &\text{ if $n \equiv 3,7$ (mod $10$)}. 
\end{cases}$$
{\em (d) ($k=6$)} We have $$\msum(n,6)=\begin{cases}
3 \;\; &\text{ if }n \equiv 1,5 \text{ (mod $6$)}, \\
1 \;\; &\text{ if }n \equiv 0,2,3,4 \text{ (mod $6$)}.
\end{cases}$$
\end{cor}

\begin{rem}
We must remark that $\msum(21,3)=2$ is not our result. This comes from \cite{Y}. 
Moreover, $\msum(6,3), \msum(9,3)=\msum(9,6)$ and $\msum(15,3)$ can be obtained 
only from Theorem \ref{thm:AFG} (i) and \eqref{chuui}. 
In fact, we have $1/2 \leq \msum(6,3) \leq \disc(6,3) = 1/2$ and the similar inequalities hold for $\msum(9,3)$ and $\msum(15,3)$. 
Note that $\disc(15,3)=1$ is mentioned in \cite{MM}. 
\end{rem}

\begin{rem}\label{rima-ku}
From Corollaries \ref{meidai} and \ref{previous}, we see that $\msum(n,3)$ coincides with $\disc(n,3)$ for any $n$. 
However, we emphasize that $\msum(n,4)$ is not necessarily equal to $\disc(n,4)$. 
In fact, we have $\msum(4m+2,4)=1$ and $\disc(4m+2,4)=2$ for any $m \geq 2$. See Proposition \ref{disc_msum}. 
\end{rem}

\section*{Acknowledgements}
The authors would like to thank Tomoki Yamashita for the proof of Lemmas \ref{hodai:mk+1} and \ref{hodai:mk-1} 
that are come from the personal communication with him in the case $k=3$. 
Originally he told the authors the main problem of this paper and the references \cite{AFG} and \cite{MM}. 
The authors would also like to thank Akitoshi Kawamura, 
who pointed out that the desired examples in Section \ref{n=-1} can be obtained from the examples given in Section \ref{n=1}. 
The authors would also like to thank anonymous referees for their careful readings and many helpful comments which make this paper more readable. 

The first author is partially supported by JSPS Grant-in-Aid for Young Scientists (B) $\sharp$17K14177. 

\bigskip


\section{Preliminaries and Known results}\label{preliminary}

In this section, we first prove two propositions. 
We also collect the known results from \cite{AFG} and \cite{MM} concerning $\disc(n,k)$ and $\msum(n,k)$. 
As a corollary of those previous results, we can see the values of $\disc(n,k)$ for some particular cases (see Corollary \ref{previous}). 
At last, we give a proof of Corollary \ref{kei}.

\begin{prop}\label{equal}
For any $n$ and $k$, we have the equalities \begin{align*} \msum(n,n-k) = \msum(n,k) \;\text{ and }\; \disc(n,n-k)=\disc(n,k). \end{align*}
\end{prop}
\begin{proof}
We see the following: 
\begin{align*}
\msum(n,n-k) &= \min\left\{\max\left\{\sum_{j=0}^{n-k-1}\pi_{i+j} - \frac{(n-k)(n+1)}{2} : i=1,\ldots,n\right\} : \pi \in S_n \right\} \\
&= \min\left\{\max\left\{\sum_{j=0}^{n-k-1}(n+1-\pi_{i+j}) - \frac{(n-k)(n+1)}{2} : i=1,\ldots,n\right\} : \pi \in S_n \right\} \\
&= \min\left\{\max\left\{-\sum_{j=0}^{n-k-1}\pi_{i+j} + \frac{(n-k)(n+1)}{2} : i=1,\ldots,n\right\} : \pi \in S_n \right\} \\
&= \min\left\{\max\left\{-\left(\frac{n(n+1)}{2}-\sum_{j=0}^{k-1}\pi_{i+j}\right) + \frac{(n-k)(n+1)}{2} : i=1,\ldots,n\right\} : \pi \in S_n \right\} \\
&= \min\left\{\max\left\{\sum_{j=0}^{k-1}\pi_{i+j} - \frac{k(n+1)}{2} : i=1,\ldots,n\right\} : \pi \in S_n \right\} \\
&= \msum(n,k). 
\end{align*}
The equality $\disc(n,n-k)=\disc(n,k)$ can be proved in the same way as above. 
\end{proof}

\begin{prop}\label{disc_msum} We have $ \disc(4m+2,4) = 2$ for any $m \geq 2$. \end{prop}
\begin{proof}
By Theorem \ref{thm:AFG} (v) below, we know that $\disc(4m+2,4) \leq 2$. Since $\disc(4m+2,4)$ is a positive integer by \eqref{chuui}, 
it will be enough to show $\disc(4m+2,4) \neq 1$. 

Suppose that $\disc(4m+2,4)=1$. Since the average of $4$-consecutive sums is $8m+6$, 
there exists $\pi=(\pi_1, \ldots, \pi_{4m+2}) \in S_{4m+2}$ such that $s_i \in \{8m+5,8m+6,8m+7\}$ for any $i=1,\ldots,4m+2$. 
Moreover, since we have 
\begin{align*}
\sum_{i=1}^{4m+2}\pi_i
=\pi_j+\sum_{r=0}^{\ell-1}s_{j+1+4r}+\pi_{j+4\ell+1}+\sum_{r=0}^{m-\ell-1}s_{j+4\ell+2+4r}
\end{align*}
for any $\ell=0,1,\ldots,m$ and $8m+5 \leq s_i \leq 8m+7$ for each $i$, we see that 
\begin{align*}
\sum_{i=1}^{4m+2}\pi_i - m(8m+7) \leq \pi_i + \pi_{i+4\ell+1} \leq \sum_{i=1}^{4m+2}\pi_i - m(8m+5) 
\end{align*}
for any $\ell=0,1,\ldots,m$. Hence, \begin{align}\label{cond_4m+2}3m+3 \leq \pi_i+\pi_{i+4\ell+1} \leq 5m+3.\end{align}

Without loss of generality, we may set $\pi_1 = 1$. Then it follows from \eqref{cond_4m+2} that $3m+3 \leq \pi_1+\pi_{4\ell+2} \leq 5m+3$, 
so we have $$\{\pi_{4\ell+2} : \ell=0,1,\ldots,m \} = \{3m+2, 3m+3, \ldots, 4m+2\}.$$ 
For $p=0,1,\ldots,m$, let $j_p$ be the index such that $\pi_{4j_p+2} = 4m+2-p$. 
We may assume that $j_0 \neq m$, otherwise we may reverse the order of $\pi_i$'s and set $j_0 = 0$. Let 
\begin{align*}
S &= \{ \pi_{4(\ell+j_0)+3} : \ell=0,1,\ldots,m\} \\
&= \{ \pi_{4\ell+3} : \ell=j_0,j_0+1,\ldots,m-1\} \cup \{ \pi_{4\ell+1} : \ell=0,1,\ldots,j_0\}. 
\end{align*}
Since $\pi_{4j_0+2}=4m+2$, we see from \eqref{cond_4m+2} that $ S = \{ 1,2,\ldots,m+1 \}$.
Now we claim the following: 
\begin{equation}\begin{split}\label{claim_4m+2}
\text{for each }p=1,2,\ldots,m, \text{ we have } 
\pi_{4j_p+1}&=m+1+p \text{ if } j_0<j_p \leq m, \text{ or}\\
\pi_{4j_p+3}&=m+1+p \text{ if } 0 \leq j_p<j_0. 
\end{split}\end{equation} 
We prove this statement by induction on $p$. When $p=1$, since $\pi_{4j_1+2}=4m+1$, we see from \eqref{cond_4m+2} that 
both $1 \leq \pi_{4j_1+1} \leq m+2$ and $1 \leq \pi_{4j_1+3} \leq m+2$. 
\begin{itemize}
\item If $j_0 < j_1 \leq m$, then $\pi_{4j_1+1} \notin S = \{ 1,2,\ldots,m+1 \}$, so we have $\pi_{4j_1+1} = m+2$. 
\item If $0 \leq j_1 < j_0$, then $\pi_{4j_1+3} \notin S = \{ 1,2,\ldots,m+1 \}$, so we have $\pi_{4j_1+3} = m+2$. 
\end{itemize}
Assume that \eqref{claim_4m+2} holds for each $1 \leq p \leq t$. 
By \eqref{cond_4m+2}, we see that $1 \leq \pi_{4j_{t+1}+1} \leq m+t+2$ and $1 \leq \pi_{4j_{t+1}+3} \leq m+t+2$. 
\begin{itemize}
\item If $j_0 < j_{t+1} \leq m$, then 
$$\pi_{4j_{t+1}+1} \notin S \cup \bigcup_{p=1}^t(\{ \pi_{4j_p+1} : j_0 < j_p \leq m\} \cup \{ \pi_{4j_p+3} : 0 \leq j_p < j_0\}) = \{ 1,2,\ldots,m+t+1 \},$$ 
so we have $\pi_{4j_{t+1}+1} = m+t+2$. 
\item If $0 \leq j_{t+1} < j_0$, then 
$$\pi_{4j_{t+1}+3} \notin S \cup \bigcup_{p=1}^t(\{ \pi_{4j_p+1} : j_0 < j_p \leq m\} \cup \{ \pi_{4j_p+3} : 0 \leq j_p < j_0\}) = \{ 1,2,\ldots,m+t+1 \},$$ 
so we have $\pi_{4j_{t+1}+3} = m+t+2$. 
\end{itemize}

The statement \eqref{claim_4m+2} can be rephrased by 
\begin{align}\label{sum_4m+2}
\text{either }\pi_{4j_p+1}+\pi_{4j_p+2}=5m+3 \text{ or } \pi_{4j_p+2}+\pi_{4j_p+3}=5m+3 \text{ holds for }p=1,2,\ldots,m. 
\end{align}
When $j_0=0$, i.e., $\pi_2=4m+2$, we see from \eqref{sum_4m+2} that 
\begin{align*}
\pi_{4m+1}+\pi_{4m+2}+\pi_1+\pi_2 &= (5m+3)+1+4m+2 = 9m+6, 
\end{align*}
but $9m+6 > 8m+7$ by $m \geq 2$, a contradiction. When $j_0 \neq 0$, we see from \eqref{sum_4m+2} that 
\begin{align*}
(\pi_{4m+1}+\pi_{4m+2})+\pi_1 = (\pi_3+\pi_2)+\pi_1 = 5m+4. 
\end{align*}
Then we have $\pi_{4m} \in \{3m+1,3m+2,3m+3\}$ and $\pi_4 \in \{3m+1,3m+2,3m+3\}$. 
However, $\pi_4, \pi_{4m} \notin \{ \pi_{4\ell+2} : \ell=0,1,\ldots,m \} = \{3m+2,3m+3, \ldots, 4m+2\}$. 
Thus $\pi_4 = \pi_{4m} = 3m+1$, a contradiction. 
\end{proof}
\begin{rem}
In \cite{S}, explicit computations show $\disc(4m+2,4)=2$ for $2 \leq m \leq 13$. 
Moreover, it is also showed in \cite{S} that $\disc(6m \pm 2,4)=2$ for $2 \leq m \leq 6$, 
$\disc(6m \pm 2,6)=2$ for $2 \leq m \leq 6$ and $\disc(8m \pm 2,8)=2$ for $2 \leq m \leq 4$. 
\end{rem}

\bigskip

We collect the results from \cite{AFG} and \cite{MM}. Most of them are the results on $\disc(n,k)$ 
but those can be used to give an upper bound for $\msum(n,k)$. 
\begin{thm}[{\cite{AFG}}]\label{thm:AFG}
We have the following: 
\begin{itemize}
\item[(i)] Let $k$ be odd. Then $\disc(2k,k)=1/2$ and $\disc(3k,k)=1$ (\cite[Theorem 2]{AFG}). 
\item[(ii)] For $n \geq 3$, we have $\disc(n,2)=1$ (\cite[Theorem 3]{AFG}). 
\item[(iii)] For $n \geq 6$, we have $\disc(n,3) \leq 2$ (\cite[Theorem 4]{AFG}). 
\item[(iv)] If $k$ is even, then $\disc(mk,k) = 1$. If $k$ is odd, then $\disc(mk,k) \leq 2$ (\cite[Theorem 7]{AFG}). 
\item[(v)] Let $g=\gcd(n,k)$ and assume that $g>1$. 
Then $\disc(n,k) \leq 2$ if $g$ is even and $\disc(n,k) \leq 7/2$ if $g$ is odd (\cite[Theorem 9]{AFG}). 
\item[(vi)] Let $g=\gcd(n,k)$ and assume that $g>1$ and $g$ is odd. Then $\disc(n,k) \leq \disc(n/g,k/g)$ (\cite[Theorem 10]{AFG}). 
\item[(vii)] Assume that $\gcd(n,k)=1$. Let $r$ be the integer with $1 \leq r \leq k-1$ such that $n \equiv r$ (mod $k$) and 
let $s$ be the smallest positive integer such that $rs \equiv \pm 1$ (mod $k$). Then $\disc(n,k) \geq k/2s$ (\cite[Theorem 11]{AFG}).  
\item[(viii)] Let $k$ be even and assume that $n \equiv \pm 1$ (mod $k$). Then $\disc(n,k)=k/2$ (\cite[Corollary 12]{AFG}). 
\end{itemize}
\end{thm}

\begin{thm}[\cite{MM}]\label{thm:MM}
We have the following: 
\begin{itemize}
\item[(i)] If $k$ is odd and $t > 1$, then $\disc(2kt,k)=\msum(2kt,k)=3/2$ (\cite[Corollay 2]{MM}). 
\item[(ii)] Let $n=6t+3 > 15$. Then $\disc(n,3)=2$ (\cite[Theorem 3]{MM}). 
\item[(iii)] If $m>2$, then $\msum(2km \pm 2, 2k)=1$ (\cite[Theorem 4]{MM}). 
\end{itemize}
\end{thm}

We collect the results which follow from Theorems \ref{thm:AFG} and \ref{thm:MM}. 
\begin{cor}\label{previous}
{\em (a) ($k=3$)} We have $$\disc(n,3)=\begin{cases}
1/2 \;\; &\text{ if }n = 6, \\
1 &\text{ if }n = 9,15, \\
3/2 &\text{ if $n$ is even with $n \geq 8$}, \\
2 &\text{ if $n$ is odd with $n \geq 7$ all but $n=9,15$}. 
\end{cases}$$
{\em (b) ($k=4$)} We have $$\disc(n,4)\begin{cases}
=2 \;\;&\text{ if $n$ is odd}, \\
=1 &\text{ if $n \equiv 0$ (mod $4$)}, \\
\leq 2 &\text{ if $n \equiv 2$ (mod $4$)}. 
\end{cases}$$
{\em (c) ($k=5$)} We have $$\disc(n,5)\begin{cases}
=1/2 &\text{ if $n=10$}, \\
=1 &\text{ if $n=15$}, \\
=3/2 &\text{ if $n \equiv 0$ (mod $10$) with $n \geq 20$}, \\
\leq 2 &\text{ if $n \equiv 5$ (mod $10$)}, \\
\geq 5/2 &\text{ if $n \equiv \pm 1$ (mod $5$)}, \\
\geq 5/4 &\text{ if $n \equiv \pm 2$ (mod $5$)}.
\end{cases}$$
{\em (d) ($k=6$)} We have $$\disc(n,6)\begin{cases}
=3 \;\;&\text{ if $n \equiv 1,5$ (mod $6$)}, \\
=1 &\text{ if $n \equiv 0,3$ (mod $6$)}, \\
\leq 2 &\text{ if $n \equiv 2,4$ (mod $6$)}. 
\end{cases}$$
\end{cor}
\begin{proof}
(a) $\disc(6,3)=1/2$ and $\disc(9,3)=1$ come from Theorem \ref{thm:AFG} (i). 
$\disc(15,3)=1$ is mentioned in the end of the proof of \cite[Theorem 3]{MM}. 

Let $n$ be even. Then it follows from Theorem \ref{thm:AFG} (iii) and \eqref{chuui} that we have $\disc(n,3) \leq 3/2$. 
When $n \not\equiv 0$ (mod $3$), i.e., $n \equiv \pm 1$ (mod $3$), we see that $\disc(n,3) \geq 3/2$ by Theorem \ref{thm:AFG} (vii). 
Hence, $\disc(n,3)=3/2$ if $n \not\equiv 0$ (mod $3$). 
When $n \equiv 0$ (mod $3$), we directly obtain $\disc(n,3)=3/2$ from Theorem \ref{thm:MM} (i). 

Let $n$ be odd with $n \geq 7$ and $n \neq 9,15$. Similar to the above discussion, 
we can see from Theorem \ref{thm:AFG} (iii), (vii) and \eqref{chuui} that $\disc(n,3)=2$ if $n \not\equiv 0$ (mod $3$). 
When $n \equiv 0$ (mod $3$), we obtain from Theorem \ref{thm:MM} (ii) that $\disc(n,3)=2$. 

(b) Let $n$ be odd, i.e., $n \equiv \pm 1$ (mod $4$). Then we have $\disc(n,4)=2$ by Theorem \ref{thm:AFG} (viii). 
Let $n \equiv 0$ (mod $4$). Then $\disc(n,4)=1$ by Theorem \ref{thm:AFG} (iv). Let $n \equiv 2$ (mod $4$). 
Then $\gcd(n,4)=2$, so we obtain $\disc(n,4) \leq 2$ by Theorem \ref{thm:AFG} (v). 

(c) $\disc(10,5)=1/2$ and $\disc(15,5)=1$ come from Theorem \ref{thm:AFG} (i). 
By Theorem \ref{thm:AFG} (iv), we see that $\disc(n,5) \leq 2$ if $n \equiv 5$ (mod $10$). 
The inequalities $\disc(n,5) \geq 5/2$ if $n \equiv \pm 1$ (mod $5$) and $\disc(n,5) \geq 5/4$ if $n \equiv \pm 2$ (mod $5$) 
follow from Theomre \ref{thm:AFG} (vii). 

(d) The cases $n \equiv \pm 1$ and $n \equiv 0$ (mod $6$) directly follow from Theorems \ref{thm:AFG} (viii) and (iv), respectively. 
Let $n \equiv 3$ (mod $6$). Then we see that $\gcd(n,6)=3$. Thus, it follows from Theorems \ref{thm:AFG} (ii) and (vi) 
together with \eqref{chuui} that $1 \leq \disc(n,6) \leq \disc(n/3, 2)=1$. 
The inequality $\msum(n,6) \leq 2$ in the case $n \equiv 2,4$ (mod $6$) follows from Theorem \ref{thm:AFG} (v) since $\gcd(n,6)=2$. 
\end{proof}

\begin{proof}[Proof of Corollary \ref{kei}] 
In the case $n$ is even, $\msum(n,k)=1$ directly follows from Theorem \ref{main2}. 

Assume that $n$ is odd. Let $g=\gcd(n,k)$. Then $g=k/2$ and $g$ is odd. Thus it follows from 
Theorems \ref{thm:AFG} (ii) and (vi) together with \eqref{chuui} that 
$$1 \leq \msum(n,k) \leq \disc(n,k) \leq \disc(n/g,k/g) = \disc(n/g,2)=1.$$ 
Hence, we conclude that $\msum(n,k)=\disc(n,k)=1$, as required. 
\end{proof}

\bigskip

\section{Proof of Theorem \ref{main1}: The case $n \equiv 1$ (mod $k$)}\label{n=1}
This section is devoted to giving a proof of Theorem \ref{main1} in the case $n \equiv 1$ (mod $k$). 
Throughout this section, let $n=mk+1$ with $m>1$. 

\begin{lem}\label{hodai:mk+1}
We have $\displaystyle \msum(mk+1,k) \geq \frac{k}{2}$. 
\end{lem}
\begin{proof}
Take $\pi=(\pi_1,\ldots,\pi_{mk+1}) \in S_n$ arbitrarily. Without loss of generality, we may set $\pi_{mk+1}=1$. Then 
$$\sum_{i=0}^{m-1}s_{ik+1}=\sum_{i=1}^{mk}\pi_i=\frac{(mk+2)(mk+1)}{2} - 1=\frac{(mk)^2 + 3mk}{2}.$$ 
Thus the average of $s_1,s_{k+1},\ldots,s_{(m-1)k+1}$ is equal to $\displaystyle \frac{(mk)^2 + 3mk}{2} \cdot \frac{1}{m} = \frac{k(mk+2)}{2} + \frac{k}{2}$. 
Hence, we obtain that \begin{align*}
\msum(mk+1,k) \geq \min\{\max\{s_1,s_{k+1},\ldots,s_{(m-1)k+1}\} : \pi \in S_n\} - \frac{k(n+1)}{2} \geq \frac{k}{2}, 
\end{align*}
as desired. 
\end{proof}

By Lemma \ref{hodai:mk+1}, it will be sufficient to prove that $\msum(mk+1,k) \leq (k+1)/2$ or $k/2$ 
with respect to the parity of $n$ and $k$. (See \eqref{chuui}.) 
When $k$ is even, the inequality $\msum(mk+1,k) \leq k/2$ follows from Theorem \ref{thm:AFG} (viii) since $\msum(n,k) \leq \disc(n,k)$. 
Thus, we may concentrate on the case $k$ is odd. Moreover, by Theorem \ref{thm:AFG} (iii), we may also assume $k \geq 5$. 

For proving $\msum(mk+1,k) \leq (k+1)/2$ when $m$ is even or $\msum(mk+1,k) \leq k/2$ when $m$ is odd, 
it will be sufficient to show the existence of $\pi \in S_n$ with $\msum(\pi,k)=(k+1)/2$ when $m$ is even or $k/2$ when $m$ is odd, respectively. 

For proving that the example of $\pi \in S_n$ enjoys the required property, the use of the following notation is convenient. 

\begin{Notation}
For $\pi \in S_n$, we define $d_i = \pi_{i+k} - \pi_i$ and ${\bf d}(\pi)=(d_1,d_2,\ldots,d_n)$. 
Using this ${\bf d}(\pi)$, we can analyze the rise and fall of $k$-consecutive sums for $\pi$. 
Regarding ${\bf d}(\pi)$, we use the notation, e.g., ${\bf d}(\pi)=(1,(1,-1)^3,-1)$, which stands for ${\bf d}(\pi) = (1,1,-1,1,-1,1,-1,-1)$.
Note that $\sum_{i=1}^n d_i=0$. 
\end{Notation}

\bigskip

In the case $m$ is even, let 
\begin{align}\label{mk+1,kodd,meven}\pi_{ik+j} = \begin{cases}
  m+1-i, \;\;       &j=1, \\
  3m/2+i+2,         &j=2, \; 0 \leq i \leq m/2-1, \\
  m/2+i+2,          &j=2, \; m/2 \leq i \leq m-1, \\
  3m-2i+1,          &j=3, \; 0 \leq i \leq m/2-1, \\
  4m-2i,            &j=3, \; m/2 \leq i \leq m-1, \\
  3m+i+2,           &j=4, \\
  (j-1)m+2+i,       &j=5,7,\ldots,k, \\
  jm+1-i,           &j=6,8,\ldots,k-1 
\end{cases}\end{align}
for $i=0,1,\ldots,m-1$ and let $\pi_{mk+1} = 1$. For seeing $\pi_1,\ldots,\pi_{mk}$, we may read the numbers off from the following configuration 
from left to right of the first row, then the second row, and through the last row: $$\begin{pmatrix} 
m+1    &3m/2+2         &3m+1   &3m+2   &4m+2   &6m+1   &\cdots &\cdots &(k-1)m+2 \\ 
m      &3m/2+3         &3m-1   &3m+3   &4m+3   &6m     &\cdots &\cdots &(k-1)m+3 \\ 
\vdots &\vdots         &\vdots &\vdots &\vdots &\vdots &       & &\vdots \\
\vdots &2m+1           &2m+3   &\vdots &\vdots &\vdots &       & &\vdots \\
\vdots &m+2            &3m     &\vdots &\vdots &\vdots &       & &\vdots \\
\vdots &m+3            &3m-2   &\vdots &\vdots &\vdots &       & &\vdots \\
\vdots &\vdots         &\vdots &\vdots &\vdots &\vdots &       & &\vdots \\
2      &3m/2+1         &2m+2   &4m+1   &5m+1   &5m+2   &\cdots &\cdots &km+1 \end{pmatrix}$$
For this $\pi$, one has $\displaystyle s_1 = \frac{k(n+1)}{2} + \frac{k+1}{2}$. 
Let $k'=k-5$. From the above configuration, we can see that 
\begin{align*}
{\bf d}(\pi) = (&(-1,1,-2,1,1,(-1,1)^{k'})^{(m-2)/2},-1,-m+1,m-3,1,1,(-1,1)^{k'}, \\
&(-1,1,-2,1,1,(-1,1)^{k'})^{(m-2)/2},-1,\lambda_1,\ldots,\lambda_{k-1},\lambda'), 
\end{align*}
where $\lambda_i < 0$ and $\lambda'>0$. Note that $s_j=s_1+\sum_{i=1}^{j-1}d_i$. 
Since we have that $\sum_{i=1}^{j-1}d_i \leq 0$ for any $j$, we conclude that $\msum(\pi,k) = (k+1)/2$. 

In the case $m$ is odd, let 
\begin{align}\label{mk+1,kodd,modd}\pi_{ik+j} =
\begin{cases}
  m+1-i, \;\;       &j=1, \\
  (3m+3)/2+i,       &j=2, \; 0 \leq i \leq (m-1)/2, \\
  (m+3)/2+i,         &j=2, \; (m+1)/2 \leq i \leq m-1, \\
  3m-2i+1,          &j=3, \; 0 \leq i \leq (m-1)/2, \\
  4m-2i+1,          &j=3, \; (m+1)/2 \leq i \leq m-1, \\
  3m+i+2,           &j=4, \\
  (j-1)m+2+i,       &j=5,7,\ldots,k, \\
  jm+1-i,           &j=6,8,\ldots,k-1 
\end{cases}\end{align}
for $i=0,1,\ldots,m-1$ and let $\pi_{mk+1} = 1$. 
For $\pi_1,\ldots,\pi_{mk}$, we may read the numbers off from the following configuration 
from left to right of the first row, then the second row, and through the last row: $$\begin{pmatrix} 
m+1    &(3m+3)/2       &3m+1   &3m+2   &4m+2   &6m+1   &\cdots &\cdots &(k-1)m+2 \\ 
m      &(3m+5)/2       &3m-1   &3m+3   &4m+3   &6m     &\cdots &\cdots &(k-1)m+3 \\ 
\vdots &\vdots         &\vdots &\vdots &\vdots &\vdots &       & &\vdots \\
\vdots &2m+1           &2m+2   &\vdots &\vdots &\vdots &       & &\vdots \\
\vdots &m+2            &3m     &\vdots &\vdots &\vdots &       & &\vdots \\
\vdots &m+3            &3m-2   &\vdots &\vdots &\vdots &       & &\vdots \\
\vdots &\vdots         &\vdots &\vdots &\vdots &\vdots &       & &\vdots \\
2      &(3m+1)/2       &2m+3   &4m+1   &5m+1   &5m+2   &\cdots &\cdots &km+1 \end{pmatrix}$$
For this $\pi$, one has $\displaystyle s_1 = \frac{k(n+1)}{2} + \frac{k}{2}$. 
Let $k'=k-5$. From the above configuration, we can see that 
\begin{align*}
{\bf d}(\pi) = (&(-1,1,-2,1,1,(-1,1)^{k'})^{(m-2)/2},-1,-m+1,m-2,1,1,(-1,1)^{k'}, \\
&(-1,1,-2,1,1,(-1,1)^{k'})^{(m-2)/2},-1,\lambda_1,\ldots,\lambda_{k-1},\lambda'), 
\end{align*}
where $\lambda_i < 0$ and $\lambda'>0$. Note that $s_j=s_1+\sum_{i=1}^{j-1}d_i$. 
Since we have that $\sum_{i=1}^{j-1}d_i \leq 0$ for any $j$, we conclude that $\msum(\pi,k) = k/2$.

\begin{ex}
(a) Let $k=7$ and $m=6$, i.e., $n=43$. Note that $k(n+1)/2=154$. Then the above $\pi$ is given like 
\begin{align*}\pi=(&7, 11, 19, 20, 26, 37, 38, 6, 12, 17, 21, 27, 36, 39, 5, 13, 15, 22, 28, 35, 40, \\
&4, 8, 18, 23, 29, 34, 41, 3, 9, 16, 24, 30, 33, 42, 2, 10, 14, 25, 31, 32, 43, 1).\end{align*}
Note that this $\pi$ can be seen from the configuration 
$\begin{pmatrix} 7 &11 &19 &20 &26 &37 &38 \\ 6 &12 &17 &21 &27 &36 &39 \\ 5 &13 &15 &22 &28 &35 &40 \\ 
4 &8 &18 &23 &29 &34 &41 \\ 3 &9 &16 &24 &30 &33 &42 \\ 2 &10 &14 &25 &31 &32 &43 \end{pmatrix}$. 
We may read this configuration from left to right of the first row, and to the second row, and so on. 
We see that the maximal $7$-consecutive sum is $158$. \\
(b) Let $k=5$ and $m=5$, i.e., $n=26$. Note that $k(n+1)/2=135/2$. Then the above $\pi$ is given like 
$$\pi=(6, 9, 16, 17, 22, 5, 10, 14, 18, 23, 4, 11, 12, 19, 24, 3, 7, 15, 20, 25, 2, 8, 13, 21, 26, 1).$$
Note that this $\pi$ can be seen from the configuration 
$\begin{pmatrix} 6 &9 &16 &17 &22 \\ 5 &10 &14 &18 &23 \\ 4 &11 &12 &19 &24 \\ 3 &7 &15 &20 &25 \\ 2 &8 &13 &21 &26 \end{pmatrix}$. 
We may read this configuration from left to right of the first row, and to the second row, and so on. 
We see that the maximal $5$-consecutive sum is $70$. 
\end{ex}

\bigskip


\section{Proof of Theorem \ref{main1}: The case $n \equiv -1$ (mod $k$)}\label{n=-1}
This section is devoted to giving a proof of Theorem \ref{main1} in the case $n \equiv -1$ (mod $k$). 
Throughout this section, let $n=mk-1$ with $m > 2$.

\begin{lem}\label{hodai:mk-1}
We have $\displaystyle \msum(mk-1,k) \geq \frac{k}{2}$. 
\end{lem}
\begin{proof}
Given an arbitrary $\pi = ( \pi_1 , \cdots , \pi_{mk-1} ) \in S_n$, we may set $\pi_1 = mk-1$ without loss of generality. 

Suppose that $\msum(mk-1,k) < k/2$. In other words, there exists a permutation $\pi$ such that
\begin{align}\label{eqs_i}s_i \leq \frac{k(n+1)}{2}+\frac{k-1}{2} = \frac{mk^2+k-1}{2}\end{align}
for each $i$. 

Since $s_1=\sum_{i=1}^k \pi_i \leq (mk^2+k-1)/2$ and $s_{(m-1)k+1}=\sum_{i=(m-1)k+1}^{mk-1} \pi_i+\pi_1 \leq (mk^2+k-1)/2$, we see that 
$$\pi_{(m-1)k+1} + \cdots + \pi_{mk-1} + \pi_1 + \cdots + \pi_k \leq mk^2 + k - 1 - (mk -1)= mk(k-1) + k. $$ 
Hence, \begin{align*}
\sum_{i=1}^{m-2} s_{ik+1} &= \frac{n(n+1)}{2} - ( \pi_{(m-1)k+1} + \cdots + \pi_k ) \geq \frac{mk(mk-1)}{2} - mk(k-1) - k \\
&= \left(\frac{mk^2}{2} + \frac{k}{2}\right) ( m-2 ). 
\end{align*}
Therefore, the average of $s_{k+1},s_{2k+1},\cdots,s_{(m-2)k+1}$ is greater than or equal to $\displaystyle \frac{mk^2}{2} + \frac{k}{2}$, 
a contradiction because of \eqref{eqs_i}. 
\end{proof}

By Lemma \ref{hodai:mk-1}, it will be sufficient to prove that $\msum(mk-1,k) \leq (k+1)/2$ or $k/2$ 
with respect to the parity of $n$ and $k$. (See \eqref{chuui}.) 
When $k$ is even, the inequality $\msum(mk-1,k) \leq k/2$ follows from Theorem \ref{thm:AFG} (viii). 
Thus, we may concentrate on the case $k$ is odd. Moreover, by Theorem \ref{thm:AFG} (iii), we may also assume $k \geq 5$.

Let $k$ be odd. Let $\pi=(\pi_1,\ldots,\pi_{mk+1}) \in S_{mk+1}$ given in \eqref{mk+1,kodd,meven} or \eqref{mk+1,kodd,modd}, 
and define $\pi'=(\pi_1',\ldots,\pi_{mk-1}') \in S_n$ by setting $\pi_i'=\pi_i-1$ for each $1 \leq i \leq mk-1$. 
In what follows, we show that this $\pi' \in S_{mk-1}$ is a desired permutation. 
Let ${\bf d}(\pi) = (d_1,\ldots,d_{mk+1})$ and let ${\bf d}(\pi') = (d_1',\ldots,d_{mk-1}')$. 
By the definition of $\pi'$, we can see that the first $((m-1)k-1)$ elements of ${\bf d}(\pi')$ coincide with those of ${\bf d}(\pi)$. 
Moreover, we can also see that $d_{(m-1)k}'<0$. 
From these discussions, we obtain that ${\bf d}(\pi') = (d_1,\ldots,d_{(m-1)k-1},d_{(m-1)k}',\lambda_1,\ldots,\lambda_{k-1})$, 
where $d_{(m-1)k}'<0$ and $\lambda_1,\ldots,\lambda_{k-1}>0$. Note that $s_j=s_1+\sum_{i=1}^{j-1}d_i$. 
Hence we conclude that $\displaystyle \msum(\pi',k) = (k+1)/2$ if $m$ is even 
(i.e., in the case \eqref{mk+1,kodd,meven}) and $k/2$ if $m$ is odd (i.e., in the case \eqref{mk+1,kodd,modd}).

\begin{ex}
(a) Let $k=7$ and $m=6$, i.e., $n=41$. Note that $k(n+1)/2=147$. Then the above $\pi'$ is given like 
\begin{align*}\pi'=(&6, 10, 18, 19, 25, 36, 37, 5, 11, 16, 20, 26, 35, 38, 4, 12, 14, 21, 27, 34, 39, \\
&3, 7, 17, 22, 28, 33, 40, 2, 8, 15, 23, 29, 32, 41, 1, 9, 13, 24, 30, 31).\end{align*}
Note that this $\pi$ can be seen from the configuration 
$\begin{pmatrix} 6 &10 &18 &19 &25 &36 &37 \\ 5 &11 &16 &20 &26 &35 &38 \\ 4 &12 &14 &21 &27 &34 &39 \\ 
3 &7 &17 &22 &28 &33 &40 \\ 2 &8 &15 &23 &29 &32 &41 \\ 1 &9 &13 &24 &30 &31 & \end{pmatrix}$. 
We see that the maximal $7$-consecutive sum is $151$. \\
(b) Let $k=5$ and $m=5$, i.e., $n=24$. Note that $k(n+1)/2=125/2$. Then the above $\pi'$ is given like 
$$\pi'=(5, 8, 15, 16, 21, 4, 9, 13, 17, 22, 3, 10, 11, 18, 23, 2, 6, 14, 19, 24, 1, 7, 12, 20).$$
Note that this $\pi$ can be seen from the configuration 
$\begin{pmatrix} 5 &8 &15 &16 &21 \\ 4 &9 &13 &17 &22 \\ 3 &10 &11 &18 &23 \\ 2 &6 &14 &19 &24 \\ 1 &7 &12 &20 & \end{pmatrix}$. 
We see that the maximal $5$-consecutive sum is $65$. 
\end{ex}

\bigskip


\section{Proof of Theorem \ref{main1}: The case $n \equiv 0$ (mod $k$)}\label{n=0}
This section is devoted to giving a proof of Theorem \ref{main1} in the case $n \equiv 0$ (mod $k$). 
Throughout this section, let $n=mk$ with $m>1$ and assume that $k$ is odd.

\begin{lem}\label{max}
Let $a_1,\ldots,a_m$ be distinct nonnegative integers 
and let $$\alpha = |\{i \in \{1,\ldots,m\}: a_{i-1} < a_i > a_{i+1}\}|,$$ where we let $a_0=a_m$ and $a_{m+1}=a_1$. 
Then we have the following: 
\begin{itemize}
\item[(i)] \; 
$\displaystyle \sum_{i=1}^m\max\{a_i,a_{i+1}\} \geq \sum_{i=1}^ma_i + \max\{a_1,\ldots,a_m\} - \min\{a_1,\ldots,a_m\} + (\alpha-1)$;
\item[(ii)] \; $\displaystyle \sum_{i=0}^{m-1}|a_i-a_{i+1}| \geq 2(m-1)+2(\alpha - 1) \geq 2(m-1)$. 
\end{itemize}
\end{lem}
\begin{proof}
(i) Let $a_1=\max\{a_1,\ldots,a_m\}$, and assume the following: 
$$a_1 > \ldots > a_{j_1} < \ldots < a_{j'_1} > \ldots > a_{j_2} < \ldots < a_{j'_2} > \ldots \ldots  > a_{j_{\alpha-1}} < \ldots < a_{j'_{\alpha-1}} > \ldots > a_{j_\alpha} < \ldots < a_m < a_1.$$
Let $a_{j_p} = \min\{a_1,\ldots,a_m\}$. Then we see that 
\begin{align*}
\sum_{i=1}^m\max\{a_i,a_{i+1}\} 
&= \sum_{i=1}^m a_i - \sum_{1 \leq i \leq \alpha, i \neq p} a_{j_i} -a_{j_p} + \sum_{1 \leq i \leq \alpha-1}a_{j'_i} + a_1 \\
&= \sum_{i=1}^m a_i + a_1 - a_{j_p} + \sum_{1 \leq i \leq \alpha-1}a_{j'_i}- \sum_{1 \leq i \leq \alpha, i \neq p} a_{j_i} \\
&= \sum_{i=1}^m a_i + a_1 - a_{j_p} + \sum_{1 \leq i \leq p-1}(a_{j'_i}-a_{j_i})+ \sum_{p+1 \leq i \leq \alpha}(a_{j'_{i-1}}-a_{j_i}) \\
&\geq \sum_{i=1}^m a_i + \max\{a_1,\ldots,a_m\} - \min\{a_1,\ldots,a_m\} + (\alpha-1). 
\end{align*}
(ii) By using (i), we see that 
\begin{align*}
2\sum_{i=1}^m\max\{a_i,a_{i+1}\} &\geq 2\left(\sum_{i=1}^m a_i + \max\{a_1,\ldots,a_m\} - \min\{a_1,\ldots,a_m\} + (\alpha-1)\right) \\
&\geq 2\sum_{i=1}^m a_i + 2(m-1) + 2(\alpha - 1), \text{ and } \\
2\sum_{i=1}^m a_i=\sum_{i=1}^m(a_i+a_{i+1}) &= \sum_{i=1}^m\max\{a_i,a_{i+1}\}+\sum_{i=1}^m\min\{a_i,a_{i+1}\}. 
\end{align*}
Hence, we obtain that 
$$\sum_{i=0}^{m-1}|a_i-a_{i+1}|=\sum_{i=1}^m\max\{a_i,a_{i+1}\} - \sum_{i=1}^m\min\{a_i,a_{i+1}\} \geq 2(m-1) + 2(\alpha - 1).$$ 
\end{proof}

\begin{lem}\label{hodai1}
We have $\displaystyle \msum(mk,k) \geq 1 - \frac{1}{m}$. 
\end{lem}
\begin{proof}
Fix $\pi=(\pi_1,\ldots,\pi_n) \in S_n$. We relabel it by $$\pi=(\pi_{1,1},\pi_{1,2},\ldots,\pi_{1,k},\pi_{2,1},\ldots,\pi_{m,k}).$$ 
Then, for each $i=1,\ldots,m$, we see that \begin{align*}
\msum(\pi,k) + \frac{k(n+1)}{2} &\geq \max\left\{\sum_{j=1}^k\pi_{i,j}, \sum_{j=2}^k\pi_{i,j}+\pi_{i+1,1}\right\}
=\sum_{j=2}^k\pi_{i,j}+\max\{\pi_{i,1},\pi_{i+1,1}\}, 
\end{align*}
where we let $\pi_{m+1,j}=\pi_{1,j}$. Hence, \begin{align*}
m\cdot\msum(\pi,k) + \frac{n(n+1)}{2} &\geq \sum_{i=1}^m \left(\sum_{j=2}^k\pi_{i,j}+\max\{\pi_{i,1},\pi_{i+1,1}\}\right) \\
&\geq \sum_{i=1}^m \sum_{j=2}^k \pi_{i,j}+\sum_{i=1}^m\pi_{i,1}+\max\{\pi_{1,1},\ldots,\pi_{m,1}\}-\min\{\pi_{1,1},\ldots,\pi_{m,1}\} \\
&=\frac{n(n+1)}{2}+\max\{\pi_{1,1},\ldots,\pi_{m,1}\}-\min\{\pi_{1,1},\ldots,\pi_{m,1}\}. 
\end{align*}
Note that the second inequality follows from Lemma \ref{max} (i). 
This implies that $$\msum(mk,k) \geq \min\left\{\frac{\max\{\pi_{1,1},\ldots,\pi_{m,1}\}-\min\{\pi_{1,1},\ldots,\pi_{m,1}\}}{m} : \pi \in S_n\right\} \geq \frac{m-1}{m},$$
as required. 
\end{proof}


Recall that $k$ is odd. When $m$ is even with $m \geq 4$ (resp. $m=2$), we see that 
$3/2 \leq \msum(mk,k) \leq \disc(mk,k) \leq 3/2$ (resp. $1/2 \leq \msum(2k,k) \leq \disc(2k,k)=1/2$) 
by Lemma \ref{hodai1} and Theorem \ref{thm:AFG} (iv) (resp. Theorem \ref{thm:AFG} (i)). 
Thus, we obtain that $\msum(mk,k)=\disc(mk,k)=3/2$ when $m$ is even with $m \geq 4$ (resp. $\msum(2k,k)=1/2$). 

Hence, our remaining case is that $m$ is odd with $m \geq 3$. We know by Theorem \ref{thm:AFG} (iv) that $\msum(mk,k) \leq \disc(mk,k) \leq 2$, 
while a lower bound can be seen from the following: 
\begin{lem}
Let $m$ and $k$ be odd with $k \geq 3$ and $m \geq 2k+3$. Then $\msum(mk,k) \geq 2$. 
\end{lem}
\begin{proof}
Given $\pi=(\pi_1,\ldots,\pi_{mk}) \in S_{mk}$, let 
$\alpha_j=|\{ \pi_{ik+j} : \pi_{(i-1)k+j} < \pi_{ik+j} > \pi_{(i+1)k+j}\}|$ for $0 \leq i \leq m-1$. 
Note that for any $j$, we have $\alpha_j \geq 1$. 
Since $\pi_j,\pi_{k+j},\ldots,\pi_{(m-1)k+j}$ are all distinct integers, it follows from Lemma \ref{max} (ii) that 
\begin{align*}
\sum_{i=0}^{m-1}|s_{ik+j}-s_{ik+j+1}|=\sum_{i=0}^{m-1}|\pi_{ik+j}-\pi_{(i+1)k+j}| \geq 2(m-1) + 2(\alpha_j-1). 
\end{align*}

Suppose that $\msum(mk,k) < 2$, i.e., $\msum(mk,k) =1$. Then \begin{align*}
|s_{ik+j}-s_{ik+j+1}|=\max\{s_{ik+j},s_{ik+j+1}\}-\min\{s_{ik+j},s_{ik+j+1}\} \leq \frac{k(mk+1)}{2}+1-\min\{s_{ik+j},s_{ik+j+1}\}. 
\end{align*}
Thus, \begin{align*}
2(m-1)+2(\alpha_j-1) \leq \sum_{i=0}^{m-1}|s_{ik+j}-s_{ik+j+1}| 
\leq \frac{mk(mk+1)}{2}+m-\sum_{i=0}^{m-1}\min\{s_{ik+j},s_{ik+j+1}\}. 
\end{align*}
Hence, \begin{align*}
\sum_{j=1}^k\sum_{i=0}^{m-1}\min\{s_{ik+j},s_{ik+j+1}\} &\leq \sum_{j=1}^k\left( \frac{mk(mk+1)}{2}+m-2(m-1)-2(\alpha_j-1) \right) \\
&=mk \cdot \frac{k(mk+1)}{2}+4k-mk-2\sum_{j=1}^k \alpha_j. 
\end{align*}
Let $\widetilde{S}=\{\min\{s_{ik+j},s_{ik+j+1}\} : 0 \leq i \leq m-1, 1 \leq j \leq k\}$ be the multi-set and 
let $S$ be the same set as $\widetilde{S}$ but it is an ordinary set (not a multi-set). Namely, we have $S \subset \widetilde{S}$. 
Note that each $s \in \widetilde{S}$ appears once or twice in $\widetilde{S}$ (i.e., not more than twice). 
Let $S'=\{s \in S: s \text{ appears once in }\widetilde{S}\}$. Then 
\begin{align*}
\sum_{s \in S'} s \leq |S'| \cdot \frac{k(mk+1)}{2}. 
\end{align*}
Hence, from the previous two inequalities, we have
\begin{align*}
2\sum_{s \in S} s= \sum_{s \in \widetilde{S}} s + \sum_{s \in S'} s 
\leq mk \cdot \frac{k(mk+1)}{2}+4k-mk-2\sum_{j=1}^k \alpha_j + |S'| \cdot \frac{k(mk+1)}{2}.
\end{align*}
From $|\widetilde{S}|=mk$, we know $mk+|S'|=2|S|$, so we can rewrite it as follows: 
$$\sum_{s \in S} s \leq |S| \cdot \frac{k(mk+1)}{2}+2k-\frac{mk}{2}-\sum_{j=1}^k \alpha_j.$$ 
Since the average of $s_i$'s is $k(mk+1)/2$, we have 
\begin{align}\label{ccc}R \geq \frac{mk}{2}+\sum_{j=1}^k \alpha_j -2k,\end{align}
where we let $R=\left|\left\{s_i : s_i=k(mk+1)/2+1\right\}\right|$. 

Let $C=(\epsilon_1,\ldots,\epsilon_{mk})$ be the sequence 
defined by $\epsilon_i=$``$<$'' if $s_i<s_{i+1}$ or $\epsilon_i=$``$>$'' if $s_i>s_{i+1}$. 
Since $s_{i-1}<s_i>s_{i+1}$ holds if $s_i=k(mk+1)/2+1$, we see that $2R \leq f(C)$, 
where $f(X)$ is the function defined in Lemma \ref{technical} below. 
Hence, it follows from Lemma \ref{technical} that $2R \leq f(C) \leq (k-1)m+2\min\{ \alpha_1,\ldots,\alpha_k\}$. Thus, 
\begin{align*}
R \leq \frac{(k-1)m}{2}+\min\{ \alpha_1,\ldots,\alpha_k\}. 
\end{align*}

Let $\alpha_\lambda=\min\{ \alpha_1,\ldots,\alpha_k\}$. Then we see the following: 
\begin{align*}
\frac{(k-1)m}{2}+\alpha_\lambda \geq 
R &\geq \frac{(k-1)m}{2}+\alpha_\lambda+\left(\frac{m}{2}+\sum_{1 \leq j \leq k, j \neq \lambda}\alpha_j-2k\right) \quad \text{(by \eqref{ccc})}\\
&\geq \frac{(k-1)m}{2}+\alpha_\lambda+\left(\frac{m}{2}+(k-1)-2k\right) \\
&= \frac{(k-1)m}{2}+\alpha_\lambda+\frac{m}{2}-k-1. 
\end{align*}
In particular, $0 \geq m/2-k-1$. However, by our assumption, we have $m \geq 2k+3$, a contradiction.

Therefore, $\msum(mk,k) \geq 2$, as desired. \end{proof}
\begin{lem}\label{technical}
For $i=1,\ldots,k$, let $A_i = (a_{i,1} , a_{i,2} , \ldots , a_{i,m})$ be a sequence of two symbols 
and let $$P = ( a_{1,1} , a_{2,1} , \ldots , a_{k,1} , a_{1,2} , a_{2,2} , \ldots , a_{k,2} , \ldots , a_{1,m} , a_{2,m} , \ldots , a_{k,m})$$ 
be a nested sequence. For a given sequence $X$ of two symbols, let $f(X)$ be the function counting the number of runs of $X$, 
which can be also understood as $f(X)=|\{ i : x_i \neq x_{i+1}\}|$ for $X=(x_1,\ldots,x_M)$, where $x_{M+1}=x_1$. 

Assume that $k$ is odd. Then we have the inequality $$f(P) \leq (k-1)m + \min \{ f(A_i) \}_{1 \leq i \leq k}.$$ 
\end{lem}
\begin{proof}
Let $r = \min \{ f(A_i) \}_{1 \leq i \leq k}$ and let $f(A_j) = r$. Then we have 
$$|\{ ( a_{j,i} , a_{j,i+1} ) : a_{j,i} \neq a_{j,i+1} \}| = r \;\text{ and }\;
|\{ ( a_{j,i} , a_{j,i+1} ) : a_{j,i} = a_{j,i+1} \}| = m-r$$ 
for $1 \leq i \leq m$, where we let $a_{j,m+1}=a_{j,1}$.

For each $1 \leq \ell \leq m$, let 
$$P_\ell = ( a_{j,1}, a_{j+1,1} , \ldots , a_{k,1},a_{1,2},\ldots, a_{j,2} , a_{j+1,2} , \ldots , a_{j,3} , \ldots, a_{j-1,\ell+1}).$$ 
Remark that $P_m=P$. Since $k$ is odd, for each $\ell=1,\ldots,m-1$, we see the following: 
\begin{itemize}
\item If $a_{j,\ell} \neq a_{j,\ell+1}$, then $f(P_{\ell+1}) \leq f(P_\ell) + k$; 
\item If $a_{j,\ell} = a_{j,\ell+1}$, then $f(P_{\ell+1}) \leq f(P_\ell) + k - 1$. 
\end{itemize}
Therefore, $$f(P)=f(P_m) \leq kr + (k-1)(m-r) = (k-1)m+r.$$ 
\end{proof}

\bigskip


\section{Proof of Theorem \ref{main2}}\label{nkeven}
Let $n$ and $k$ be even. In this section, we prove Theorem \ref{main2}. 
Since $\msum(n,k) \geq 1$ always holds by \eqref{chuui}, it is enough to show the existence of $\pi \in S_n$ with $\msum(\pi,k)=1$. 
The remaining part of this section is devoted to constructing such $\pi$. 

Let $n=qk+r$ for some $q,r \in \ZZ$, where $q >0$ and $0 \leq r \leq k-1$. 
In the following, we will construct a configuration of numbers from $1$ to $n$ by putting each number to one of $n$ boxes. 
\begin{itemize}
\item First, we prepare $(q+1) \times k$ boxes and remove $(k-r)$ boxes from the most upper right box. 
\item Let $(a,b)$ denote the box placed in the $a$th row (from the upper) and the $b$th (from the left) one, 
where we let $(a+q+1,b)=(a,b)$ and $(a,b+k)=(a,b)$. 
\item Next, we put the numbers $1,2,\ldots,n/2$ by the following manner: 
\begin{itemize}
\item Put $1$ at $(1,2)$. 
\item Let $t>1$ and assume that $(t-1)$ is at $(a,b)$. Then we put $t$ at 
\begin{itemize}
\item $(a+1,b)$ if $a < q+1$; 
\item $(1,b+r)$ if $a=q+1$ and $1 \leq b+r \leq k-r$ when $(1,b+r)$ is empty, or $(1,b+r+2)$ when some number is already put at $(1,b+r)$; 
\item $(2,b+r)$ if $a=q+1$ and $k-r+1 \leq b+r \leq k$ when $(2,b+r)$ is empty, or $(2,b+r+2)$ when some number is already put at $(2,b+r)$. 
\end{itemize}
\end{itemize}
\item Finally, we put the numbers $n/2+1,\ldots,n$ by the following manner: 
\begin{itemize}
\item Put $n$ at $(1,1)$. 
\item Let $t<n$ and assume that $(t+1)$ is at $(a,b)$. Then we put $t$ at 
\begin{itemize}
\item $(a+1,b)$ if $a < q+1$; 
\item $(1,b+r)$ if $a=q+1$ and $1 \leq b+r \leq k-r$ when $(1,b+r)$ is empty, or $(1,b+r+2)$ when some number is already put at $(1,b+r)$; 
\item $(2,b+r)$ if $a=q+1$ and $k-r+1 \leq b+r \leq k$ when $(2,b+r)$ is empty, or $(2,b+r+2)$ when some number is already put at $(2,b+r)$. 
\end{itemize}
\end{itemize}
\end{itemize}
For example, in the case $n=48$ and $k=18$, the configuration looks as follows: 
\begin{table}[htbp]
  \begin{center}
    \begin{tabular}{|c|c|c|c|c|c|c|c|c|c|c|c|c|c|c|c|c|c|}
      \hline
      48 &  1 & 40 &  9 & 32 & 17 & 43 &  6 & 35 & 14 & 27 & 22 &$*$ &$*$ &$*$ &$*$ &$*$ &$*$ \\ \hline
      47 &  2 & 39 & 10 & 31 & 18 & 42 &  7 & 34 & 15 & 26 & 23 & 45 &  4 & 37  & 12 & 29  & 20   \\ \hline
      46 &  3 & 38 & 11 & 30 & 19 & 41 &  8 & 33 & 16 & 25 & 24 & 44 &  5 & 36  & 13 & 28  & 21   \\
      \hline
    \end{tabular}
  \end{center}
\end{table}

From this configuration, we define $\pi=(\pi_1,\ldots,\pi_n) \in S_n$ by $\pi_{ik+j}=(q+1-i,j)$ for $0 \leq i \leq q$ and $1 \leq j \leq k$. 
For the above example, $\pi$ looks like $$(\underbrace{46,3,38,11,\ldots,21}_{3\text{rd row}},
\underbrace{47,2,39,10,\ldots,20}_{2\text{nd row}},\underbrace{48,1,\ldots,22}_{1\text{st row}}).$$

Then we see that if $1 \leq j \leq k$ is odd, then $\displaystyle s_{ik+j}=k/2 \cdot (n+1)$. 
Moreover, we also see that if $j$ is even and $1 \leq i \leq q-2$, or $j$ is even with $j \leq r$ and $i=q-1$, 
then $\displaystyle s_{ik+j}=k(n+1)/2 +1$. 

Let us consider the case $i=q-1$ and $j$ is even with $r+2 \leq j \leq k$. Then 
\begin{align*}
s_{ik+j}=s_{(q-1)k+j-1}-\pi_{(q-1)k+j-1}+\pi_{(q-1)k+j-1+k} = s_{(q-1)k+j-1}-\pi_{(q-1)k+j-1}+\pi_{j-r-1}. 
\end{align*}
From the construction, we have that $\pi_{(q-1)k+j-1}$ is the number put in $(2,j-1)$ and 
$\pi_{j-r-1}$ is the number put in $(q+1,j-r-1)$. By the procedure of putting the numbers, 
we see that the number in $(q+1,j-r-1)$ is equal to the one in $(2,j-1)$ plus $1$ or less than the one in $(2,j-1)$. 
Hence we obtain that $\pi_{j-r-1}-\pi_{(q-1)k+j-1} \leq 1$. Therefore, one has $s_{ik+j} \leq s_{ik+j-1}+1=k(n+1)/2+1$. 

Let us consider the case $i=q$ and $j$ is even with $2 \leq j \leq r$. Then 
\begin{align*}
s_{ik+j}=s_{qk+j-1}-\pi_{qk+j-1}+\pi_{qk+j-1+k} = s_{qk+j-1}-\pi_{qk+j-1}+\pi_{k+j-r-1}. 
\end{align*}
From the construction, we have that $\pi_{qk+j-1}$ is the number put in $(1,k+j-1)$ and 
$\pi_{k+j-r-1}$ is the number put in $(q+1,k+j-r-1)$. Since the procedure of putting the numbers, 
we see that the number in $(q+1,k+j-r-1)$ is equal to the one in $(1,k+j-1)$ plus $1$ or less than the one in $(1,k+j-1)$. 
Hence we obtain that $\pi_{k+j-r-1}-\pi_{qk+j-1} \leq 1$. Therefore, one has $s_{ik+j} \leq s_{ik+j-1}+1=k(n+1)/2+1$, 
as required. 

\begin{flushright}$\square$\end{flushright}

\bigskip


\section{Proofs of Theorem \ref{main3} and Proposition \ref{main4}}\label{k=5}

This section is devoted to giving a proof of Theorem \ref{main3} and a proof of Proposition \ref{main4}. 

\bigskip
\begin{proof}[Proof of Theorem \ref{main3}]
We will prove that $\msum((2a+1)m+a+1,2a+1) > 1$ for any positive integers $a$ and $m$ with $a \geq 2$. . 

Let $n=(2a+1)m+a+1$. Suppose that $\msum(n,2a+1) \leq 1$. Then there is $\pi \in S_n$ with $\msum(\pi,2a+1) \leq 1$. 
Let $\pi_{a+1} = 1$. Since 
\begin{align*}
\sum_{i=1}^n \pi_i + \sum_{i=1}^a \pi_i &= \sum_{j=0}^{m}s_{1+j(2a+1)} \leq (m+1) \cdot \left(\frac{(2a+1)(n+1)}{2}+1\right), 
\end{align*}
we have $$\sum_{i=1}^a \pi_i \leq (m+1) \cdot \left(\frac{(2a+1)(n+1)}{2}+1\right) - \frac{n(n+1)}{2}.$$
Hence, 
\begin{align*}
\sum_{j=0}^{m-1}s_{a+2+j(2a+1)} = \frac{n(n+1)}{2} - \sum_{i=1}^{a+1}\pi_i 
&\geq \frac{n(n+1)}{2} - \left ( (m+1) \cdot \left(\frac{(2a+1)(n+1)}{2}+1\right) - \frac{n(n+1)}{2} +1 \right ) \\
&= n(n+1) - (m+1) \cdot \left(\frac{(2a+1)(n+1)}{2}+1\right) - 1. 
\end{align*}
However, by our assumption, we see that 
\begin{align*}
\sum_{j=0}^{m-1}s_{a+2+j(2a+1)} \leq m \cdot \left(\frac{(2a+1)(n+1)}{2}+1\right). 
\end{align*}
Therefore, 
\begin{align*}
n(n+1) - (m+1) \cdot \left(\frac{(2a+1)(n+1)}{2}+1\right) - 1 \leq m \cdot \left(\frac{(2a+1)(n+1)}{2}+1\right) 
\; \Longrightarrow \; a-2 \leq (3-2a)m
\end{align*}
Since $a \geq 2$, we have $m \leq \displaystyle \frac{a-2}{3-2a} \leq 0$, a contradiction. 
\end{proof}

\bigskip
\begin{proof}[Proof of Proposition \ref{main4}]
We divide into two cases: $n$ is even and $n$ is odd. Namely, we consider the case $n \equiv 2$ (mod $10$) and the case $n \equiv 7$ (mod $10$). 

\bigskip

First, we prove that $\msum(10m+2,5) \geq 3/2$ for any positive integer $m$. 

Suppose that $\msum(10m+2,5) < 3/2$, i.e., $\msum(10m+2,5)=1/2$ by \eqref{chuui}. 
Then there is $\pi \in S_{10m+2}$ with $\msum(\pi,5) = 1/2$. For each $i=1,\ldots,10m+2$ and $j=0,1,\ldots,2m$, we have 
\begin{align*}
\sum_{\ell=1}^{10m+2}\pi_\ell - \pi_i - \pi_{i+5j+1} = \sum_{h=0}^{j-1}s_{i+1+5h} + \sum_{h=j}^{2m-1}s_{i+2+5h} \leq 2m(25m+8). 
\end{align*}
Hence, we obtain that $$9m+3 \leq \pi_i + \pi_{i+5j+1}.$$
Let $\pi_1 = 1$. Then we have $\pi_{5j+2} \geq 9m+2$ for each $j=0,1,\ldots,2m$, thus we see that 
\begin{align*}
\{\pi_{5j+2} : j=0, 1,\ldots, 2m\} \subset \{9m+2, 9m+3, \ldots, 10m+2\}. 
\end{align*}
However, since $m > 0$, we see that 
\begin{align*}
2m+1=|\{\pi_{5j+2} : j=0,1, \ldots, 2m\}| > |\{9m+2, 9m+3, \ldots, 10m+2\}|=m+1, 
\end{align*}
a contradiction.

\bigskip

Next, we prove that $\msum(10m+7,5) \geq 2$ for any positive integer $m$.

Suppose that $\msum(10m+7,5)=1$. Then there is $\pi \in S_{10m+7}$ with $\msum(\pi,5) = 1$. 
For each $i=1,\ldots,10m+7$ and $j=0,1,\ldots,2m+1$, we have 
\begin{align*}
\sum_{\ell=1}^{10m+7}\pi_\ell - \pi_i - \pi_{i+5j+1} = \sum_{h=0}^{j-1}s_{i+1+5h} + \sum_{h=j}^{2m}s_{i+2+5h} \leq (2m+1)(5(5m+4)+1). 
\end{align*}
Hence, we obtain that $$8m+7 \leq \pi_i + \pi_{i+5j+1}.$$
Let $\pi_1 = 1$. Then we have $\pi_{5j+2} \geq 8m+6$ for each $j$. Thus, we see that 
\begin{align}\label{k1}
\{\pi_{5j+2} : j=0, 1, \ldots, 2m+1\} = \{8m+6, 8m+7, \ldots, 10m+7\}. 
\end{align}
Moreover, since $\pi_3 + \pi_4 \geq 8m+7$, we also see that 
\begin{align*}
5(5m+4)+1 \geq \pi_{10m+7} + \pi_1 + \pi_2 + \pi_3 + \pi_4 \geq  \pi_{10m+7} + 1 + \pi_2 + 8m+7. 
\end{align*}
Hence, \begin{align}\label{k2} \pi_{10m+7} + \pi_2 \leq 17m+13. \end{align}
By using our assumption, we have 
\begin{equation*}
\begin{array}{ccccccc}
5(5m+4)+1 &\geq& s_3 &=& s_2 &+& (\pi_7 - \pi_2), \\
5(5m+4)+1 &\geq& s_8 &=& s_7 &+& (\pi_{12} - \pi_7), \\
&&& \vdots &&& \\
5(5m+4)+1 &\geq& s_{10m+3} &=& s_{10m+2} &+& (\pi_{10m+7} - \pi_{10m+2}), 
\end{array}
\end{equation*}
and 
\begin{equation}\label{k3}
\begin{array}{ccccccc}
5(5m+4)+1 &\geq& s_2 &=& s_3 &-& (\pi_7 - \pi_2), \\
5(5m+4)+1 &\geq& s_7 &=& s_8 &-& (\pi_{12} - \pi_7), \\
&&& \vdots &&& \\
5(5m+4)+1 &\geq& s_{10m+2} &=& s_{10m+3} &-& (\pi_{10m+7} - \pi_{10m+2}). 
\end{array}
\end{equation}
Since $\pi_{5j+7}-\pi_{5j+2}$ is either positive or negative for each $j=0,1,\ldots,2m$, we see that 
either 
\begin{align*}
&|\{j \in \{0,1,\ldots,2m\} : s_{5j+2} \leq 5(5m+4) \}| \geq m+1\;\text{ or }\\
&|\{j \in \{0,1,\ldots,2m\} : s_{5j+3} \leq 5(5m+4) \}| \geq m+1
\end{align*} holds. 
In the former case, we see that 
\begin{align*}
(5m+4)(10m+7) &= \pi_1 + \sum_{j=0}^{2m} s_{5j+2} + \pi_{10m+7} \leq 1 + m \cdot (5(5m+4)+1) + (m+1) \cdot 5(5m+4) + \pi_{10m+7}. 
\end{align*}
Hence, $\pi_{10m+7} \geq 9m+7$. By \eqref{k1} and \eqref{k2}, we obtain that $\pi_2 = 8m+6$ and $\pi_{10m+7} = 9m+7$. 
From $\pi_2 = 8m+6$, we see that $$\sum_{j=0}^{2m}s_{5j+3}=(5m+4)(10m+7)-(\pi_1+\pi_2) = (2m+1)(5(5m+4)+1).$$
From $s_{5j+3} \leq 5(5m+4)+1$ by our assumption, we have $s_3 = s_8 = \cdots = s_{10m+3} = 5(5m+4)+1$. It follows from \eqref{k3} that 
$\pi_7-\pi_2 \geq 0, \pi_{12}-\pi_7 \geq 0, \ldots, \pi_{10m+7} - \pi_{10m+2} \geq 0$. 
Hence, $$8m+6=\pi_2 \leq \pi_7 \leq \cdots \leq \pi_{10m+7} = 9m+7,$$
a contradiction. 

Similarly, we can lead a contradiction in the latter case. 
\end{proof}

\end{document}